\documentclass[reqno]{amsart}

\allowdisplaybreaks[3]

\usepackage{amssymb}
\usepackage{amsmath}
\usepackage{amsthm}
\usepackage{tikz}
\usepackage{latexsym}
\usepackage[justification=centering]{caption}
\usepackage{color}
\usepackage{hyperref}
\hypersetup{
 colorlinks=true,
 linkcolor=blue,
 citecolor=blue}

 \tikzset{vertex/.style={fill,circle,inner sep=1.0pt}}
 
 \makeatletter
\@addtoreset{equation}{section}
\makeatother

\newtheorem{theorem}{Theorem}[section]
\newtheorem{proposition}[theorem]{Proposition}

\newtheorem{lemma}[theorem]{Lemma}
\theoremstyle{definition}
\newtheorem{definition}[theorem]{Definition}
\newtheorem{remark}[theorem]{Remark}
\newtheorem{example}[theorem]{Example}

\newcommand{\lex}[2]{{#1}[{#2}]}
\newcommand{\neib}[2]{\overline{N}_{#1} (#2)}
\newcommand{\sphere}[1]{S^{#1}}
\newcommand{\dl}[2]{\mathrm{dl}_{#1}(#2)}
\newcommand{\st}[2]{\mathrm{st}_{#1}(#2)}
\newcommand{\lk}[2]{\mathrm{lk}_{#1}(#2)}
\newcommand{\dlK}[1]{\mathrm{dl}_K (#1)}
\newcommand{\stK}[1]{\mathrm{st}_K (#1)}
\newcommand{\lkK}[1]{\mathrm{lk}_K (#1)}
\newcommand{\poly}{\mathcal{Z}^*_M (\underline{K}, \underline{L})}

\newcommand{\pair}{(\underline{K}, \underline{L})}
\newcommand{\polyX}[1]{\mathcal{Z}^*_{#1} (\underline{X})}
\newcommand{\polyIX}[1]{\mathcal{Z}^*_{I(L_{#1})}(X)}
\newcommand{\card}[1]{\# \left( #1 \right)}

\title[]{On the Homotopy Type of the Polyhedral Join over the Independence Complex of a Forest}
\author[]{Kengo Okura}
\address{Osaka Metropolitan University, 1-1 Gakuen-cho, Naka-ku, Sakai, Osaka 599-8531, Japan}
\email{okura.kengo.k35@kyoto-u.jp}
\begin{document}
\begin{abstract}
We consider a certain class of simplicial complexes which includes the independence complexes of forests. 
We show that if a simplicial complex $K$ belongs to this class, then the polyhedral join $\mathcal{Z}^*_{K}(\underline{X}, \emptyset)$ is homotopy equivalent to a wedge sum of CW complexes of the form $\Sigma^r X_{i_1} * X_{i_2} * \cdots * X_{i_k}$, where $\underline{X}$ is a family $\{X_i\}_{i \in V(K)}$ of CW complexes and $\Sigma$ denotes the unreduced suspension. 
This result is applied to study the homotopy type of the independence complex of the lexicographic product $G[H]$ of a graph $H$ over a forest $G$. 
We denote by $L_m$ a tree on $m$ vertices with no branches. 
We show that the geometric realization of the independence complex of $L_m [H]$ is homotopy equivalent to a wedge sum of spheres if $m \neq 2,3$ and the geometric realization of the independence complex of $H$ is homotopy equivalent to a wedge sum of same dimensional spheres.
\end{abstract}

\maketitle

\section{Introduction}
\label{introduction}
Let $\mathcal{C}$ be a (convenient) category of spaces or the category of abstract simplicial complexes. 
Consider an associative operation $\star: \mathcal{C} \times \mathcal{C} \to \mathcal{C}$ such that if both $i_A : A \to X$ and $i_b : B \to Y$ are the inclusion maps, then so is $i_A \star i_B: A \star B \to X \star Y$.
For an abstract simplicial complex $K$ on $[m]=\{1,2,\ldots,m\}$ and a family of pairs of spaces (complexes) $(\underline{X}, \underline{A}) = \{(X_i \supset A_i)\}_{i \in [m]}$, we can define a union $\mathcal{Z}^{\star}_{K}(\underline{X}, \underline{A})$ of subspaces (subcomplexes) of $X_1 \star X_2 \star \cdots \star X_m$ by
\begin{align*}
\mathcal{Z}^{\star}_{K}(\underline{X}, \underline{A}) = \bigcup_{\sigma \in K} Y^{\sigma}_1 \star Y^{\sigma}_2 \star \cdots \star Y^{\sigma}_m, \ \ Y^{\sigma}_i = \left\{
\begin{aligned}
&X_i & &(i \in \sigma), \\
&A_i & &(i \notin \sigma) . 
\end{aligned} \right.
\end{align*}
(In the following, if $(X_i, A_i)=(X, A)$ $(i \in [m])$ for some pair $(X, A)$, then we write $(\underline{X}, \underline{A})=(X,A)$.)
A representative example of this construction is the {\it polyhedral product} $\mathcal{Z}^{\times}_{K}(\underline{X}, \underline{A})$, which was 
defined by Bahri, Bendersky, Cohen, and Gitler \cite{BBCG10} and has been investigated by many researchers since then.

We assume that $\mathcal{C}$ is the category of finite CW complexes (or the category of abstract simplicial complexes), and that $\star$ is the {\it join} operation $*$. Then $\mathcal{Z}^{*}_{K}(\underline{X}, \underline{A})$ is the {\it polyhedral join} of the family $(\underline{X}, \underline{A})$ over $K$, which was defined by Ayzenberg \cite{Ayzenberg13}.
In this paper, we study the homotopy type of the polyhedral joins of finite CW complexes over the simplicial complex $K$ which meets a certain condition. This condition is a stronger version of {\it vertex-decomposability}, introduced by Bj{\"{o}}rner and Wachs \cite{BjornerWachs97}, and so we name it {\it s-vertex-decomposability}. We obtain the following decomposition. Here, we set
\begin{align*}
X^{*k} &= \left\{
\begin{aligned}
&\underbrace{X * X * \cdots *X}_{k} & & ( k >0) \\
& \emptyset & &(k =0),
\end{aligned} \right. \\
\Sigma^r X &= X * S^{r-1} ,
\end{align*}
where $S^d$ denotes the $d$-dimensional sphere.
\begin{theorem}
\label{s-vertex-decomposable splitting}
Let $K$ be a connected non-empty s-vertex-decomposable simplicial complex on $[m]$ and $\underline{X} = \{X_i\}_{i \in [m]}$ be a family of finite CW complexes. Then we have
\begin{align*}
\mathcal{Z}^*_{K}(\underline{X}, \emptyset) \simeq \bigvee_{\alpha \in A} \left( \Sigma^{r_\alpha} X_1^{*k_{1, \alpha}} * \cdots * X_m^{*k_{m,\alpha}} \right)
\end{align*}
for a family $\{ (r_\alpha, k_{1, \alpha}, \ldots, k_{m, \alpha}) \}_{\alpha \in A}$ of tuples of $m+1$ non-negative integers.
\end{theorem}

We apply results on polyhedral joins to {\it independence complexes} of graphs. In this paper, a {\it graph} $G$ always means a finite undirected graph with no multiple edges and loops. Its vertex set and edge set are denoted by $V(G)$ and $E(G)$, respectively.
A subset $\sigma$ of $V(G)$ is an {\it independent set} if any two vertices of $\sigma$ are not adjacent. The independent sets of $G$ are closed under taking subsets, so they form an abstract simplicial complex. We call this abstract simplicial complex the independence complex of $G$ and denote by $I(G)$. Namely
\begin{align*}
I(G) = \{ \sigma \subset V(G) \ |\ uv \notin E(G) \text{ for any $u, v \in \sigma$ } \}.
\end{align*}
Independence complexes of graphs are no less important than other simplicial complexes constructed from graphs and have been studied in many contexts. 

A lot of research on the geometric realization of the independence complexes of graphs has been going on since Kozlov \cite{Kozlov99} determined the homotopy type of $|I(L_n)|$ and $|I(C_n)|$, where $L_m$ is a tree on $m$ vertices with no branches, and $C_n$ is a cycle on $n$ vertices ($n \geq 3$). 
One of the tools which are often used in previous studies is {\it discrete Morse theory}. It is a method introduced by Forman \cite{Forman98} and reformulated by Chari \cite{Chari00}. 
Bousquet-M{\'{e}}lou, Linusson, and Nevo \cite{BousquetmelouLinussonNevo08} and Thapper \cite{Thapper08} studied the independence complexes of {\it grid graphs} by performing discrete Morse theory as a combinatorial algorithm called a {\it matching tree}. 
However, it is hard to distinguish two CW complexes that have the same number of cells in each dimension only by discrete Morse theory. 
We need topological approaches in case discrete Morse theory is not available. For example, it is effective to represent an independence complex of a graph as a union of independence complexes of subgraphs, as in Engstr{\"{o}}m \cite{Engstrom09}, Adamaszek \cite{Adamaszek12}, and Barmak \cite{Barmak13}. 
Iriye \cite{Iriye12}, who investigated the independence complexes of some grid graphs by using these methods simultaneously, conjectured that the geometric realizations of the independence complexes of any {\it cylindrical square grid graphs} are homotopy equivalent to wedge sums of spheres.

As Harary \cite{Harary69} mentioned, there are various ways to construct a graph structure on $V(G_1) \times V(G_2)$ for given two graphs $G_1$ and $G_2$. A cylindrical square grid graph is the {\it Cartesian product} of $L_m$ and $C_n$. 
It seems natural to consider other products of $L_m$ and $C_n$ and investigate the geometric realizations of their independence complexes.
In this paper, we are interested in the {\it lexicographic product}.
The lexicographic product $\lex{G}{H}$ can be regarded to have $\card{V(G)}$ pseudo-vertices. Each of them is isomorphic to $H$ and two pseudo-vertices are ``adjacent'' if the corresponding vertices of $G$ are adjacent. Independence complexes of lexicographic products are studied by Vander Meulen and Van Tuyl \cite{VandermeulenVantuyl17} from a combinatorial point of view.
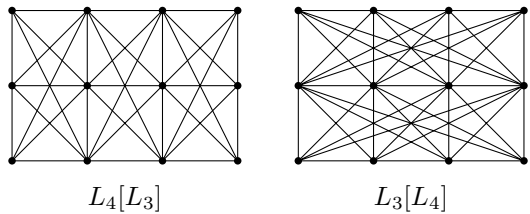
\begin{figure}[thb]
\begin{tabular}{ccc}
\begin{tikzpicture}
\draw (1,1) grid (4,3);
\draw (1,1)--(3,3) (1,2)--(2,3) (2,1)--(4,3) (3,1)--(4,2) (1,3)--(3,1) (1,2)--(2,1) (2,3)--(4,1) (3,3)--(4,2);
\draw (1,1)--(2,3) (1,3)--(2,1) (2,1)--(3,3) (2,3)--(3,1) (3,1)--(4,3) (3,3)--(4,1);
\foreach \x in {1,2,3,4} {\foreach \y in {1,2,3 } {\node at (\x, \y) [vertex] {};};}
\node at (2.5,0.5) {$\lex{L_4}{L_3}$};
\end{tikzpicture}
& &
\begin{tikzpicture}
\draw (1,1) grid (4,3);
\draw (1,1)--(3,3) (1,2)--(2,3) (2,1)--(4,3) (3,1)--(4,2) (1,3)--(3,1) (1,2)--(2,1) (2,3)--(4,1) (3,3)--(4,2);
\draw (1,1)--(3,2) (1,1)--(4,2) (2,1)--(4,2) (3,1)--(1,2) (4,1)--(1,2) (4,1)--(2,2) (1,2)--(3,3) (1,2)--(4,3) (2,2)--(4,3) (3,2)--(1,3) (4,2)--(1,3) (4,2)--(2,3);
\foreach \x in {1,2,3,4} {\foreach \y in {1,2,3 } {\node at (\x, \y) [vertex] {};};}
\node at (2.5,0.5) {$\lex{L_3}{L_4}$};
\end{tikzpicture}
\end{tabular}
\caption{Lexicographic products $\lex{L_4}{L_3}$ and $\lex{L_3}{L_4}$.}
\end{figure}

The following claims correlate our study on the polyhedral join over s-vertex-decomposable simplicial complex and the study on the independence complex of the lexicographic product over a forest.
\begin{itemize}
\item For two graphs $G, H$, we have 
\begin{align}
\label{lexico poly}
|I(\lex{G}{H})| \cong \mathcal{Z}^*_{I(G)} ( |I(H)|, \emptyset ).
\end{align}
\item If a graph $G$ is a forest, then $I(G)$ is s-vertex-decomposable.
\end{itemize}

\noindent
Using these claims, we can deduce the following result on the independence complex of the lexicographic product from Theorem \ref{s-vertex-decomposable splitting}.

\begin{theorem}
\label{forest wedge of spheres}
Let $G$ be a forest and $H$ be a graph. 
Suppose that $G$ has no vertex $u$ such that $uv \in E(G)$ for any $v \in V(G) \setminus \{u\}$ and that $|I(H)|$ is homotopy equivalent to a wedge sum of spheres. 
Then, $|I(\lex{G}{H})|$ is homotopy equivalent to a wedge sum of spheres.
\end{theorem}
\noindent
There is various previous research on graphs such that the geometric realizations of their independence complexes are homotopy equivalent to wedge sums of spheres. 
For example, the result of Van Tuyl and Villarreal \cite[Theorem 2.13]{VantuylVillarreal08} implies that if $G$ is a {\it chordal graph}, then $|I(G)|$ is either contractible or homotopy equivalent to a wedge sum of spheres. (Later, Woodroofe \cite[Theorem 1]{Woodroofe09} refined this result.) This fact was also proved by Kawamura \cite[Theorem 1.1]{Kawamura10}. 

We also present the following explicit calculation.
\begin{theorem}
\label{LnX homotopy}
Let $X$ be a finite CW complex. Then we have
\begin{align*}
\mathcal{Z}^*_{I(L_n)} ( X, \emptyset) \simeq \left\{
\begin{aligned}
& X & &(n=1), \\
&X \sqcup X& &(n=2) ,\\
&X^{*2} \sqcup X & &(n=3), \\
&\bigvee_{k, r \in \mathbb{N}_{\geq 0}} \left( {\bigvee}_{\binom{k+1}{n-2k-3r+1} \binom{k+r}{r}} \Sigma^r X^{*k} \right) & &(n \geq 4).
\end{aligned} \right.
\end{align*}
\end{theorem}
\noindent
Here, $\binom{p}{q}$ denotes the binomial coefficient. We define $\binom{p}{q}=0$ if $q<0$ or $p<q$. So, the number of the wedge summand is finite.

This calculation deduces the following result.
\begin{theorem}
\label{line theorem}
If $|I(H)| \simeq {\bigvee}_{n} S^k$, then
\begin{align*}
|I(\lex{L_m}{H})|  \simeq &\left\{
\begin{aligned}
&{\bigvee}_n \sphere{k} & &(m=1), \\
&\left( {\bigvee}_n \sphere{k} \right) \sqcup \left( {\bigvee}_n \sphere{k} \right) & &(m=2), \\
&\left( {\bigvee}_{n^2} \sphere{2k+1} \right) \sqcup \left( {\bigvee}_n \sphere{k} \right) & &(m=3), \\
&\bigvee_{d \geq 0} \left( {\bigvee}_{\sum_{p \geq 0} n^p \binom{p+1}{3(d-pk+1)-m} \binom{d-pk+1}{p} } \sphere{d} \right)& &(m \geq 4).
\end{aligned} \right.
\end{align*}
\end{theorem}

\noindent
Even though this is precisely the case in which we cannot use discrete Morse theory, we successfully determine the homotopy type of $|I(\lex{L_m}{H})|$.
With Kozlov's computation \cite[Proposition 5.2]{Kozlov99}, we can determine the homotopy types of $|I(\lex{L_m}{C_n})|$ for any $m \geq 1$ and $n \geq 3$ by Theorem \ref{line theorem}.
This result is in contrast to the case of cylindrical square grid graphs, in which we cannot easily determine the homotopy type for large $m,n$.

\section{Preliminaries}
For a positive integer $m$, we set $[m] = \{1,2,\ldots, m\}$.
We denote the one-point space by $\mathrm{pt}$ and the $d$-dimensional sphere by $S^d$.

Let $(X_1, x_1), \ldots, (X_n, x_n)$ be pointed spaces. The {\it wedge sum} $\bigvee_{i \in [n]} X_i$ of these pointed spaces is defined by
\begin{align*}
\bigvee_{i \in [n]} X_i = \left. \bigsqcup_{i \in [n]} X_i \middle\slash (x_i \sim x_j \text{ for any $i, j \in [n]$} ) . \right. 
\end{align*}
If $(X_1, x_1), \ldots, (X_n, x_n)$ are CW pairs and $X_i$ is connected, then the replacement of the base point $x_i$ with another $0$-cell $y_i$ does not change the homotopy type of the wedge sum. Hence, the wedge sum $X_1 \vee \cdots \vee X_n$ of connected non-empty CW complexes $X_1, \ldots, X_n$ is well-defined up to homotopy. 
We denote the wedges sum of $n$ copies of $X$ by ${\bigvee}_n X$. Namely
\begin{align*}
 {\bigvee}_n X = \left\{
\begin{aligned}
&\underbrace{X \vee \cdots \vee X}_{n} & &(n \geq 1), \\
&\mathrm{pt} & &(n=0).
\end{aligned} \right.
\end{align*}
Though the $0$-sphere $S^0$ is not connected, we use the following notations:
\begin{align*}
 {\bigvee}_n S^0 &= \underbrace{ \mathrm{pt} \sqcup \cdots \sqcup \mathrm{pt}}_{n+1}, \\
X \vee \left( {\bigvee}_n S^0 \right) &=X \sqcup \underbrace{ \mathrm{pt} \sqcup \cdots \sqcup \mathrm{pt}}_{n}.
\end{align*}

In this paper, we deal with finite CW complexes. For two CW complexes, it is known that the {\it join} $X*Y$ has a CW structure with the cells being the product cells of $X \times Y \times (0,1)$ or the cells of $X$ or $Y$ (see \cite{Hatcher01}).
We use the following properties of joins without proofs.
\begin{proposition}
\label{join property}
Let $X, Y$ be finite CW complexes and $X_1, X_2 \subset X$, $Y_1, Y_2 \subset Y$ be subcomplexes.
\begin{enumerate}
\item $X_1 * Y_1$ is a subcomplex of $X*Y$.
\item we have
\begin{align*}
(X_1 \cup X_2) *Y &\cong (X_1 * Y) \cup (X_2 *Y), \\
(X_1 \cap X_2) * (Y_1 \cap Y_2) &\cong (X_1 * Y_1) \cap (X_2 * Y_2).
\end{align*} 
\end{enumerate}
\end{proposition}

For a space $X$, the {\it unreduced suspension} of $X$, denoted by $\Sigma X$, is defined by
\begin{align*}
\Sigma X = X * S^0 .
\end{align*}
Unreduced suspension often appears in this paper since the following lemma plays an important role.
\begin{lemma}
\label{mapping cylinder}
Let $X$ be a finite CW complex and $X_1, X_2$ be subcomplexes of $X$ such that $X=X_1 \cup X_2$. If the inclusion maps $i_1: X_1\cap X_2 \to X_1$ and $i_2 : X_1 \cap X_2 \to X_2$ are null-homotopic, then we have
\begin{align*}
X \simeq X_1 \cup_{\mathrm{pt}} X_2 \cup_{\mathrm{pt}} \Sigma (X_1 \cap X_2) .
\end{align*}
In addition, suppose that $X_1$ and $X_2$ are connected. Then we have
\begin{align*}
X \simeq X_1 \vee X_2 \vee \Sigma (X_1 \cap X_2) .
\end{align*}
\end{lemma}
\begin{proof}
We have
\begin{align*}
X_1 \cup X_2 &\simeq (X_1 \cup X_2) \times [0,1] \\
&\simeq X_1 \times \{0\} \cup (X_1 \cap X_2) \times [0,1] \cup X_2 \times \{1\}
\end{align*}
since $(X_1, X_1 \cap X_2)$ and $(X_2, X_1 \cap X_2)$ are CW pairs.

Let $u \in X_1$ and $v \in X_2$ be points such that $i_1 \simeq c_u$ and $i_2 \simeq c_v$, where $c_u : X_1 \cap X_2 \to X_1$ and $c_v :X_1 \cap  X_2 \to X_2$ are the constant maps to $u$ and $v$, respectively. Then, since $(X, X_1)$ and $(X, X_2)$ are CW pairs, we have
\begin{align*}
&X_1 \times \{0\} \cup (X_1 \cap X_2) \times [0,1] \cup X_2 \times \{1\} \\ 
\simeq &X_1 \cup_{\{(u,0)\}} (X_1 \cap X_2) * \{(u,0),(v,1)\} \cup_{\{(v,1)\}} X_2 \\
= &X_1 \cup_{\{(u,0)\}} \Sigma (X_1 \cup X_2) \cup_{\{(v,1)\}} X_2.
\end{align*}
This is the desired conclusion.
\end{proof}

\begin{proposition}
\label{disjoint suspension}
Let $X, Y, Z$ be finite non-empty CW complexes. Then we have
\begin{align*}
(X \sqcup Y) *Z \simeq (X*Z) \vee (Y*Z) \vee \Sigma Z.
\end{align*}
In particular, we have
\begin{align*}
\Sigma(X \sqcup Y) \simeq \Sigma X \vee \Sigma Y \vee \sphere{1}.
\end{align*}
\end{proposition}
\begin{proof}
By Proposition \ref{join property}, we have
\begin{align*}
(X \sqcup Y) *Z \cong (X*Z) \cup_{\emptyset * Z} (Y*Z).
\end{align*}
We obtain the desired conclusion by Lemma \ref{mapping cylinder} since two inclusion maps $\emptyset * Z \to X *Z$, $\emptyset * Z \to Y*Z$ are null-homotopic.
\end{proof}

\begin{proposition}
\label{join wedge}
Let $X, Y, Z$ be non-empty finite CW complexes. Then we have
\begin{align*}
(X \vee Y) *Z \simeq (X*Z) \vee (Y*Z).
\end{align*}
In particular, we have
\begin{align*}
\Sigma (X \vee Y) \simeq \Sigma X \vee \Sigma Y.
\end{align*}
\end{proposition}
\begin{proof}
By Proposition \ref{join property}, we have
\begin{align*}
(X \vee Y) *Z \cong (X*Z) \cup_{\mathrm{pt} * Z} (Y*Z).
\end{align*}
We obtain the desired conclusion by Lemma \ref{mapping cylinder} since $\mathrm{pt} *Z$ is contractible.
\end{proof}

A {\it finite abstract simplicial complex} $K$ is a collection of finite subsets of a given finite set $V(K)$ such that
if $\sigma \in K$ and $\tau \subset \sigma$, then $\tau \in K$. In this paper, we drop the adjectives ``finite'' and ``abstract''. 
An element of $K$ is called a {\it simplex} of $K$. 
For a simplex $\sigma$ of $K$, we set $\dim \sigma = \card{\sigma} -1 $, where $\card{\sigma}$ is the cardinality of $\sigma$. 
The {\it geometric realization} of $K$ is denoted by $|K|$.

Let $V$ be a finite set and $F_1, \ldots, F_t \subset V$ be a collection of subsets of $V$. Define a simplicial complex $\langle F_1, \ldots, F_t \rangle$ on $V$ by
\begin{align*}
\langle F_1, \ldots, F_t \rangle= \{ \sigma \subset V \ |\ \sigma \subset F_i \text{ for some } i \in [t] \}.
\end{align*}
In particular, $\langle [n+1] \rangle$ is called {\it $n$-simplex} and denoted by $\Delta^n$.

$L \subset K$ is called a {\it subcomplex} of $K$ if $L$ is a simplicial complex. For a simplicial complex $K$ and a vertex $v$ of $K$, we define subcomplexes $\stK{v}, \dlK{v}, \lkK{v}$ of $K$ by
\begin{align*}
\stK{v} &= \{ \sigma \in K \ |\ \sigma \cup \{v\} \in K \}, \\
\dlK{v} &= \{ \sigma \in K \ |\ v \notin \sigma \}, \\
\lkK{v} &= \{ \sigma \in K \ |\ \sigma \cup \{v \} \in K, v \notin \sigma \}.
\end{align*}
By definition, we have
\begin{align}
\label{st lk dl decomposition}
K = \stK{v} \cup_{\lkK{v}} \dlK{v}.
\end{align}

Let $\{K_i\}_{i \in [m]}$ be a family of simplicial complexes. The {\it join} of $K_1, \ldots, K_m$, denoted by $K_1 * \cdots * K_m$, is a simplicial complex defined by
\begin{align}
\label{simplicial join}
K_1 * \cdots * K_m = \left\{ \sigma \subset \bigsqcup_{i \in [m]} V(K_i) \ \middle|\ 
\begin{aligned}
&\sigma \cap V(K_i) \in K_i \\
&\text{ for any $i \in [m]$}  
\end{aligned}
\right\}.
\end{align} 

A {\it finite undirected simple graph} $G$ is a pair $(V(G), E(G))$, where $V(G)$ is a finite set and $E(G)$ is a subset of $2^{V(G)}$ such that $\card{e}=2$ for any $e \in E(G)$. An element of $V(G)$ is called a {\it vertex} of $G$, and an element of $E(G)$ is called an {\it edge} of $G$. In order to indicate that $e=\{u, v\}$ ($u,v \in V(G)$), we write $e =uv$. In this paper, we drop the adjectives ``finite'', ``undirected'', and ``simple'', and call $G$ a {\it graph}.

For a vertex $v \in V(G)$, an {\it open neighborhood} $N_G (v)$ of $v$ in $G$ is defined by 
\begin{align*}
N_G (v) = \{ u \in V(G) \ |\ uv \in E(G) \}. 
\end{align*}
A {\it closed neighborhood} $\neib{G}{v}$ of $v$ in $G$ is defined by $\neib{G}{v} = N_G (v) \sqcup \{ v\}$.

A {\it full subgraph} $H$ of a graph $G$ is a graph such that 
\begin{align*}
V(H) &\subset V(G), \\
E(H) &=\{ uv \in E(G) \ |\ u, v \in V(H) \}.
\end{align*}
For two full subgraphs $H, K$ of $G$, a full subgraph whose vertex set is $V(H) \cap V(K)$ is denoted by $H \cap K$, and a full subgraph whose vertex set is $V(H) \setminus V(K)$ is denoted by $H \setminus K$. 
If $V(G)=V(H) \sqcup V(K)$ and $uv \notin E(G)$ for any $u \in V(H)$ and $v \in V(K)$, we write $G = H \sqcup K$.
For a subset $U \subset V(G)$, $G \setminus U$ is the full subgraph of $G$ such that $V(G \setminus U) = V(G) \setminus U$.

Let $L_n$ be a tree on $n$ vertices with no branches, and $C_n$ be a cycle on $n$ vertices ($n \geq 3$). Namely
\begin{align*}
&V(L_n)=\{1,2,\ldots, n\}, & &E(L_n) = \{ij  \ |\ |i-j|=1 \} , \\
&V(C_n) = \{1,2, \ldots, n \}, & &E(C_n) = E(L_n) \cup \{n1 \}.
\end{align*}
A graph $G$ is a {\it forest} if $G$ does not contain $C_n$ as a full subgraph for any $n \geq 3$.

Recall that the {\it independence complex} $I(G)$ of a graph $G$ is a simplicial complex defined by
\begin{align*}
I(G) = \{ \sigma \subset V(G) \ |\ uv \notin E(G) \text{ for any $u, v \in \sigma$ } \}.
\end{align*}
For a full subgraph $H$ of $G$, $I(H)$ is a subcomplex of $I(G)$. Furthermore, if $H, K$ are full subgraphs of $G$, then $I(H \cap K) = I(H) \cap I(K)$.

Let $G$ be a graph and $v$ be a vertex of $G$. By definition, we have
\begin{align*}
\st{I(G)}{v} &= I(G \setminus N_G(v)), \\
\dl{I(G)}{v} &= I(G \setminus \{v\}), \\
\lk{I(G)}{v} &= I(G \setminus \neib{G}{v} ) .
\end{align*}

\section{S-Vertex-Decomposable Simplicial Complex}
In this section, we define a class of simplicial complexes, which we call {\it s-vertex-decomposable} simplicial complexes.
\begin{definition}
Let $K$ be a simplicial complex. We say that $K$ is {\it s-vertex-decomposable} if
\begin{itemize}
\item $K$ is either a simplex or $\{ \emptyset\}$, or
\item there exist vertices $v, w$ of $K$ such that
\begin{itemize}
\item $\dlK{v}$, $\lkK{v}$ are s-vertex-decomposable, and 
\item $\dlK{v} = \stK{w}$.
\end{itemize}
We call $v$ an {\it s-shedding vertex} of $K$.
\end{itemize} 
\end{definition}

An important example of s-vertex-decomposable simplicial complex is the independence complex of a forest.
\begin{theorem}
\label{forest s-v-d}
Let $G$ be a graph. If $G$ is a forest, then $I(G)$ is s-vertex-decomposable.
\end{theorem}
\begin{proof}
We prove the theorem by induction on $\card{V(G)}$. If $\card{V(G)}=0$, then $I(G) = \{\emptyset\}$, which is s-vertex-decomposable. 

For a non-negative integer $n$, assume that $I(G')$ is s-vertex-decomposable for any forest $G'$ such that $\card{V(G')}<n$. 
Let $G$ be a forest such that $\card{V(G)} =n$. If $G$ has no edge, then $I(G)$ is s-vertex-decomposable since $I(G) = \Delta^{n-1}$. If $G$ has at least one edge, then $G$ has a leaf $l$. Let $v \in V(G)$ be a unique vertex of $G$ such that $vl \in E(G)$. Then $\dl{I(G)}{v} = I(G \setminus \{v\})$, $\lk{I(G)}{v} = I(G \setminus \neib{G}{v})$ are s-vertex-decomposable by the assumption of the induction. Furthermore, we have
\begin{align*}
\dl{I(G)}{v} = I(G \setminus \{v\}) = I(G \setminus N_G(l)) = \st{I(G)}{l} .
\end{align*}
Therefore, $I(G)$ is s-vertex-decomposable.
\end{proof}

\begin{remark}
The converse of Theorem \ref{forest s-v-d} does not hold. A counterexample is $G$ in Figure \ref{forest counter example}. 
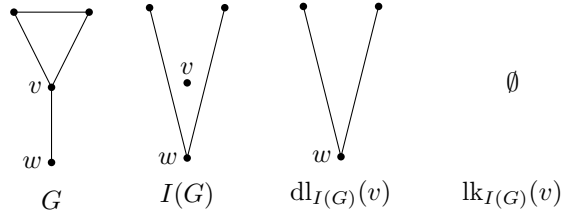
\begin{figure}[thb]
\begin{tabular}{ccccccc}
\begin{tikzpicture}
\draw (0,0)--(0.5,1)--(-0.5,1)--cycle (0,0)--(0,-1);
\foreach \x in {(0,0), (0.5,1), (-0.5,1), (0,-1)} {\node at \x [vertex] {};};
\node at (0,0) [left] {$v$};
\node at (0,-1) [left] {$w$};
\node at (0,-1.5) {$G$};
\end{tikzpicture}
& &
\begin{tikzpicture}
\draw (0.5,1)--(0,-1)--(-0.5,1);
\foreach \x in {(0,0), (0.5,1), (-0.5,1), (0,-1)} {\node at \x [vertex] {};};
\node at (0,0) [above] {$v$};
\node at (0,-1) [left] {$w$};
\node at (0,-1.5) {$I(G)$};
\end{tikzpicture}
& &
\begin{tikzpicture}
\draw (0.5,1)--(0,-1)--(-0.5,1);
\foreach \x in {(0.5,1), (-0.5,1), (0,-1)} {\node at \x [vertex] {};};
\node at (0,-1) [left] {$w$};
\node at (0,-1.5) {$\dl{I(G)}{v}$};
\end{tikzpicture}
& &
\begin{tikzpicture}
\node at (0,0) {$\emptyset$};
\node at (0,-1.5) {$\lk{I(G)}{v}$};
\end{tikzpicture}
\end{tabular}
\caption{A graph $G$ such that $I(G)$ is s-vertex-decomposable but $G$ is not a forest}
\label{forest counter example}
\end{figure}
\end{remark}

We observe that if an s-vertex-decomposable simplicial complex is not connected, then the number of connected components must be $2$, and one of the connected components is a $0$-simplex.
\begin{proposition}
\label{non-connected}
Let $K$ be a non-empty s-vertex-decomposable simplicial complex. If $K$ is not connected, then 
\begin{align*}
K = \stK{w} \sqcup \langle \{v\} \rangle,
\end{align*}
for some $v, w \in V(K)$.
\end{proposition}
\begin{proof}
We first observe that $K$ is not a simplex since $K$ is not connected. So, there exists $v, w$ such that $\dlK{v}$, $\lkK{v}$ are s-vertex-decomposable, and $\dlK{v} = \stK{w}$. Here we have
\begin{align*}
K = \stK{v} \cup_{\lkK{v}} \dlK{v} = \stK{v} \cup_{\lkK{v}} \stK{w}.
\end{align*}
Thus, if $K$ is not connected, then $\lkK{v}$ must be empty since both $\stK{v}$ and $\stK{w}$ are connected. Therefore, we get the desired conclusion since $\lkK{v} = \emptyset$ implies that $\stK{v} = \langle \{v\} \rangle$.
\end{proof}

Recall from \cite[Definition 11.1]{BjornerWachs97} that a simplicial complex $K$ is called {\it vertex decomposable} if 
\begin{itemize}
\item $K$ is either a simplex or $\{ \emptyset\}$, or
\item there exists a vertex $v$ of $K$ such that
\begin{itemize}
\item $\dlK{v}$, $\lkK{v}$ are vertex-decomposable, and 
\item no maximal simplex of $\lkK{v}$ is maximal in $\dlK{v}$.
\end{itemize}
\end{itemize}
\noindent
We show that s-vertex-decomposability is a stronger version of vertex-decomposability.
\begin{lemma}
\label{link inclusion}
Let $K$ be a simplicial complex and $v,w$ be vertices of $K$ such that $\dlK{v} = \stK{w}$. Then we have
\begin{enumerate}
\item $v \neq w$ ,
\item $\lkK{w} = \dlK{v} \cap \dlK{w}$,
\item $\lkK{v} \subset \lkK{w}$.
\end{enumerate}
\end{lemma}
\begin{proof}
We get $v \neq w$ from $v \notin \dlK{v} = \stK{w} \ni w$. We also obtain
\begin{align*}
\lkK{w} = \stK{w} \cap \dlK{w} = \dlK{v} \cap \dlK{w}.
\end{align*}
Thus, in order to prove (3), it remains to show that $\lkK{v} \subset \dlK{w}$. We have $\{ v, w\} \notin K$ since $v \notin \dlK{v} = \stK{w}$. Thus, if $\sigma \in \lkK{v}$ contains $w$, then it implies $\{v, w\} \subset \sigma \cup \{v\} \in K$, a contradiction. So, $ w \notin \sigma$ for any $\sigma \in \lkK{v}$.
\end{proof}

\begin{proposition}
Let $K$ be a simplicial complex. If $K$ is s-vertex-decomposable, then $K$ is vertex-decomposable.
\end{proposition}
\begin{proof}
Let $K$ be a s-vertex-decomposable simplicial complex. We prove the proposition by induction on $\card{V(K)}$. 
If $\card{V(K)} = 0$, then $K = \{ \emptyset \}$, which is vertex-decomposable.
 
For a non-negative integer $n$, suppose that if $K'$ is s-vertex-decomposable simplicial complex such that $\card{V(K')} < n$, then $K'$ is vertex-decomposable. Consider $K$ with $\card{V(K)} =n$.
If $K = \Delta^{n-1}$, then $K$ is vertex-decomposable. Assume that there exist vertices $v, w$ of $K$ such that
\begin{itemize}
\item $\dlK{v}$, $\lkK{v}$ are s-vertex-decomposable, and 
\item $\dlK{v} = \stK{w}$.
\end{itemize}
$\dlK{v}$ and $\lkK{v}$ are vertex-decomposable by the assumption of the induction. Furthermore, for any maximal simplex $\sigma \in \lkK{v}$, we have 
\begin{align*}
\sigma \subsetneq \sigma \cup \{w\} \in \stK{w} = \dlK{v} .
\end{align*}
The above inclusion follows from Lemma \ref{link inclusion}. So, $\sigma$ is not maximal in $\dl{K}{v}$. Hence, $K$ is vertex-decomposable.
\end{proof}

\section{Homotopy type of Polyhedral Join over S-Vertex-Decomposable Simplicial Complex}
As defined in Section \ref{introduction}, the {\it polyhedral join} of finite CW complexes (or simplicial complexes) is defined as follows.
\begin{definition}[Based on {\cite[Definition 4.2, Observation 4.3]{Ayzenberg13}}]
Let $K$ be a simplicial complex on $[m]$ and $(\underline{X}, \underline{A}) = \{(X_i, A_i)\}_{i \in [m]}$ be a family of finite CW pairs. 
For a simplex $\sigma \in K$, we define $(\underline{X}, \underline{A})^{* \sigma}$ by
\begin{align*}
(\underline{X}, \underline{A})^{* \sigma} = Y_1 * Y_2 * \cdots * Y_m, \ Y_i= \left\{
\begin{aligned}
&X_i & &(i \in \sigma) , \\
&A_i & &(i \notin \sigma) .
\end{aligned} \right.
\end{align*}
$(\underline{X}, \underline{A})^{* \sigma}$ is a subcomplex of $X_1 * \cdots * X_m$ by Proposition \ref{join property}.

$\mathcal{Z}^*_K (\underline{X}, \underline{A})$ is a subcomplex of $X_1 * \cdots *X_m$ defined by
\begin{align*}
\mathcal{Z}^*_K (\underline{X}, \underline{A}) = \bigcup_{\sigma \in K} (\underline{X}, \underline{A})^{* \sigma} 
\end{align*}
(union is taken in $X_1 * \cdots *X_m$).

If each $X_i$ is a simplicial complex and $A_i$ is a subcomplex of $X_i$, then $\mathcal{Z}^*_K (\underline{X}, \underline{A})$ is defined as a simplicial complex by (\ref{simplicial join}).
\end{definition}

In the following, we consider the polyhedral join $\mathcal{Z}^*_K (\underline{X}, \emptyset)$, where $K$ is a simplicial complex and $\underline{X}$ is a family of finite CW complexes. We denote this polyhedral join by $\polyX{K}$.

$\polyX{K}$ behaves well toward operations on the simplicial complex $K$.
\begin{proposition}
\label{union decomposition}
Let $K$ be a simplicial complex and $\underline{X}=\{X_v\}_{v \in V(K)}$ be a family of finite CW complexes.
\begin{enumerate}
\item Let $L \subset K$ be a subcomplex of $K$. Then $\polyX{L}$ is a subcomplex of $\polyX{K}$.
\item Let $L_1, L_2 \subset K$ be subcomplexes of $K$ such that $K= L_1 \cup L_2$. Then we have
\begin{align*}
\polyX{K} = \polyX{L_1} \cup_{\polyX{L_1 \cap L_2}} \polyX{L_2}.
\end{align*}

\end{enumerate}
\end{proposition}
\begin{proof}
(1) directly follows from the definition of polyhedral joins. 

We prove (2). First, we have
\begin{align*}
\polyX{K} &= \bigcup_{\sigma \in K} (\underline{X}, \emptyset)^{* \sigma} \\
&=\left(\bigcup_{\sigma \in L_1} (\underline{X}, \emptyset)^{* \sigma} \right) \cup \left( \bigcup_{\tau \in L_2} (\underline{X}, \emptyset)^{* \tau} \right) \\
&=\polyX{L_1} \cup \polyX{L_2}.
\end{align*}
Remarking that we obtain $\polyX{L_1} \cap \polyX{L_2} \supset \polyX{L_1 \cap L_2}$ from (1), it remains to show that $\polyX{L_1} \cap \polyX{L_2} \subset \polyX{L_1 \cap L_2}$. 
For a point $x \in \polyX{L_1} \cap \polyX{L_2}$, there exist $\sigma \in L_1$ and $\tau \in L_2$ such that 
\begin{align*}
x &\in (\underline{X}, \emptyset)^{* \sigma} \cap (\underline{X}, \emptyset)^{* \tau} \\
&= (A_1 * \cdots *A_m) \cap (B_1 * \cdots *B_m) & &\left( 
\begin{aligned}
&A_i = \left\{
\begin{aligned}
&X_i & &(i \in \sigma) \\
&\emptyset & &(i \notin \sigma)
\end{aligned} \right.  \\
&B_i = \left\{
\begin{aligned}
&X_i & &(i \in \tau) \\
&\emptyset & &(i \notin \tau)
\end{aligned} \right. 
\end{aligned}
\right) \\
&= (A_1 \cap B_1) * \cdots *(A_m \cap B_m) & &\left(A_i \cap B_i = \left\{
\begin{aligned}
&X_i  & &(i \in \sigma \cap \tau) \\
&\emptyset & &(i \notin \sigma \cap \tau)
\end{aligned}
\right. \right) \\
&= (\underline{X}, \emptyset)^{* (\sigma \cap \tau)}.
\end{align*}
The second equality follows from Proposition \ref{join property}.
So, we get $x \in \polyX{L_1 \cap L_2}$ since $\sigma \cap \tau \in L_1 \cap L_2$.
\end{proof}

\begin{remark}
Proposition \ref{union decomposition} does not hold for $\mathcal{Z}^*_{K} (\underline{X}, \underline{A})$ if $\underline{A} \neq \emptyset$.
For example, consider $K=\langle \{1\}, \{2\} \rangle$, $L_1 = \langle \{1\} \rangle$, $L_2= \langle \{2 \} \rangle$, and $X_1=a \sqcup b$, $A_1 = a$, $X_2 = c \sqcup d$, $A_2 = c$, where $a, b, c, d$ are points. Then
\begin{align*}
&\mathcal{Z}^*_{K} ( \underline{X}, \underline{A}) = (ac \cup bc) \cup (ac \cup ad) = bc \cup ac \cup ad , \\
&\mathcal{Z}^*_{L_1} ( \underline{X}, \underline{A}) \cup_{\mathcal{Z}^*_{L_1 \cap L_2}(\underline{X}, \underline{A})} \ \mathcal{Z}^*_{L_2} ( \underline{X}, \underline{A}) = (a \sqcup b) \sqcup (c \sqcup d) ,
\end{align*}
where $ac, bc, ad$ denote line segments (see Figure \ref{union counter example}).

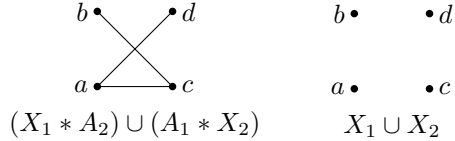
\begin{figure}[thb]
\begin{tabular}{ccc}
\begin{tikzpicture}
\node at (1,1) [vertex] {}; \node at (1,1) [left] {$a$}; 
\node at (1,2) [vertex] {}; \node at (1,2) [left] {$b$};
\node at (2,1) [vertex] {}; \node at (2,1) [right] {$c$};
\node at (2,2) [vertex] {}; \node at (2,2) [right] {$d$};
\draw (1,2)--(2,1)--(1,1)--(2,2);
\node at (1.5, 0.5) {$(X_1 * A_2) \cup (A_1 * X_2)$};
\end{tikzpicture}
& &
\begin{tikzpicture}
\node at (1,1) [vertex] {}; \node at (1,1) [left] {$a$}; 
\node at (1,2) [vertex] {}; \node at (1,2) [left] {$b$};
\node at (2,1) [vertex] {}; \node at (2,1) [right] {$c$};
\node at (2,2) [vertex] {}; \node at (2,2) [right] {$d$};
\node at (1.5, 0.5) {$X_1 \cup X_2$};
\end{tikzpicture}
\end{tabular}
\caption{An example of $\mathcal{Z}^*_{L_1 \cup L_2} (\underline{X}, \underline{A}) \neq \mathcal{Z}^*_{L_1} (\underline{X}, \underline{A}) \cup \mathcal{Z}^*_{L_2} (\underline{X}, \underline{A})$}
\label{union counter example}
\end{figure}
\end{remark}

\begin{proposition}
\label{join decomposition}
Let $K, L$ be simplicial complexes and  $\underline{X}=\{X_v\}_{v \in V(K) \sqcup V(L)}$ be a family of finite CW complexes. Then we have
\begin{align*}
\polyX{K*L} \cong \polyX{K} * \polyX{L}.
\end{align*}
\end{proposition}
\begin{proof}
Remark that each simplex of $K * L$ is a disjoint union of simplexes $\sigma \in K$ and $\tau \in L$. Thus, 
\begin{align*}
\polyX{K*L} &= \bigcup_{\sigma \sqcup \tau \in K*L} (\underline{X}, \emptyset)^{* (\sigma \sqcup \tau)} \\
&=\bigcup_{\sigma \in K, \tau \in L} \left( (\underline{X}, \emptyset)^{* \sigma} * (\underline{X}, \emptyset)^{* \tau} \right) \\
&\cong \left( \bigcup_{\sigma \in K} (\underline{X}, \emptyset)^{* \sigma} \right) * \left( \bigcup_{\tau \in L} (\underline{X}, \emptyset)^{* \tau} \right) \\
&=\polyX{K} * \polyX{L}.
\end{align*}
Note that the homeomorphism follows from Proposition \ref{join property}.
\end{proof}

Now we consider $\polyX{K}$ with $K$ s-vertex-decomposable. The following theorem is essentially important throughout this paper.
\begin{theorem}
\label{splitting}
Let $K$ be a simplicial complex. Suppose that there exist vertices $v, w$ of $K$ such that $\lkK{v} \neq \emptyset$ and $\dlK{v} = \stK{w}$.
Let $\underline{X}=\{X_v\}_{v \in V(K)}$ be a family of non-empty finite CW complexes. Then we have
\begin{align*}
\polyX{K} \simeq &\Sigma \polyX{\lkK{v}} \vee \left(\polyX{\lkK{v}} * X_v \right) \vee \polyX{\dlK{v}}  .
\end{align*}
\end{theorem}
\begin{proof}
By Proposition \ref{union decomposition} and the decomposition (\ref{st lk dl decomposition}), we have
\begin{align*}
\polyX{K} = \polyX{\stK{v}} \cup_{\polyX{\lkK{v}}} \polyX{\dlK{v}} .
\end{align*}

Let $i : \polyX{\lkK{v}} \to \polyX{\stK{v}}$ and 
$j: \polyX{\lkK{v}} \to \polyX{\dlK{v}}$ be the inclusion maps. By Proposition \ref{join decomposition}, we have
\begin{align*}
\polyX{\stK{v}} &= \polyX{\lkK{v} * \langle \{v\} \rangle}  \\
&\cong  \polyX{\lkK{v}} * X_v, \\
\polyX{\dlK{v}} &= \polyX{\stK{w}}  \\
&\cong \polyX{\lkK{w}} * X_w.
\end{align*}
Let $x \in X_v$ and $y \in X_w$ be points. Then, we have
\begin{align*}
\polyX{\lkK{v}} * \{x\} &\subset  \polyX{\lkK{v}} * X_v, \\
\polyX{\lkK{v}} * \{y\} &\subset \polyX{\lkK{w}} * \{y\} \\
&\subset \polyX{\lkK{w}} * X_w .
\end{align*}
Note that the second inclusion is obtained from Lemma \ref{link inclusion}.
These inclusions indicate that $i, j$ are null-homotopic. Therefore, by Lemma \ref{mapping cylinder}, we obtain
\begin{align*}
\polyX{K} = &\polyX{\stK{v}} \cup_{\polyX{\lkK{v}}} \polyX{\dlK{v}} \\
\simeq &\Sigma \polyX{\lkK{v}} \vee \polyX{\stK{v}} \vee \polyX{\dlK{v}} \\
\simeq &\Sigma \polyX{\lkK{v}} \vee \left(\polyX{\lkK{v}} * X_v \right) \vee \polyX{\dlK{v}}.
\end{align*} 
So, the proof is completed.
\end{proof}

\begin{remark}
The assumption ``$\dlK{v} = \stK{w}$'' in Theorem \ref{splitting} cannot be replaced with ``no maximal simplex of $\lkK{v}$ is maximal in $\dlK{v}$'', which appears in the definition of vertex-decomposability. Consider a simplicial complex $K$ and a vertex $v$ in Figure \ref{splitting counter example}.

\begin{figure}[thb]
\begin{tabular}{ccccc}
\begin{tikzpicture}
\node at (-0.865, 1) [vertex] {};
\node at (-0.865,0) [vertex] {};
\node at (0,0.5) [vertex] {};
\node at (0,-0.5) [vertex] {};
\node at (0.865, 1) [vertex] {};
\node at (0.865,0) [vertex] {};
\fill [opacity=0.3] (0,0.5)--(-0.865,1)--(-0.865,0)--(0,-0.5)--(0.865,0)--(0.865,1)--cycle;
\draw (0,0.5)--(-0.865,1)--(-0.865,0)--(0,-0.5)--(0.865,0)--(0.865,1)--cycle (0,0.5)--(-0.865,0) (0,0.5)--(0,-0.5) (0, 0.5)--(0.865, 0);
\node at (-0.865, 1) [left] {$a$};
\node at (-0.865,0) [left] {$b$};
\node at (0,0.5) [above] {$u$};
\node at (0,-0.5) [below] {$v$};
\node at (0.865, 1) [right] {$d$};
\node at (0.865,0) [right] {$c$};
\node at (0,-1) {$K$};
\end{tikzpicture}
&  &
\begin{tikzpicture}
\node at (-0.865, 1) [vertex] {};
\node at (-0.865,0) [vertex] {};
\node at (0,0.5) [vertex] {};
\node at (0.865, 1) [vertex] {};
\node at (0.865,0) [vertex] {};
\fill [opacity=0.3] (0,0.5)--(-0.865,1)--(-0.865,0)--cycle (0,0.5)--(0.865,0)--(0.865,1)--cycle;
\draw (0,0.5)--(-0.865,1)--(-0.865,0)--cycle (0,0.5)--(0.865,0)--(0.865,1)--cycle;
\node at (-0.865, 1) [left] {$a$};
\node at (-0.865,0) [left] {$b$};
\node at (0,0.5) [above] {$u$};
\node at (0.865, 1) [right] {$d$};
\node at (0.865,0) [right] {$c$};
\node at (0,-1) {$\dlK{v}$};
\end{tikzpicture}
& &
\begin{tikzpicture}
\node at (-0.865,0) [vertex] {};
\node at (0,0.5) [vertex] {};
\node at (0.865,0) [vertex] {};
\draw (0,0.5)--(-0.865,0) (0,0.5)--(0.865,0);
\node at (-0.865,0) [left] {$b$};
\node at (0,0.5) [above] {$u$};
\node at (0.865,0) [right] {$c$};
\node at (0,-1) {$\lkK{v}$};
\end{tikzpicture}
\end{tabular}
\caption{An example of $v$ that does not satisfy the assumption in Theorem \ref{splitting}}
\label{splitting counter example}
\end{figure}
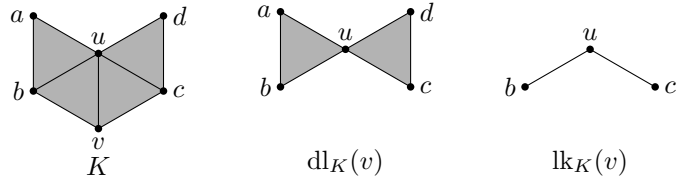

We get
\begin{align*}
\mathcal{Z}^*_{K} (\underline{X}) &\simeq (X_u * X_a *X_b) \vee (X_u * X_b * X_v) \vee (X_u * X_v * X_c) \\*
&\  \vee (X_u * X_c * X_d) \vee \Sigma(X_u * X_b) \vee \Sigma(X_u * X_v) \vee \Sigma(X_u * X_c) , \\
\mathcal{Z}^*_{\dlK{v}} (\underline{X}) &\simeq (X_u * X_a * X_b) \vee (X_u * X_c *X_d) \vee \Sigma X_u , \\
\mathcal{Z}^*_{\lkK{v}} (\underline{X}) &\simeq (X_u * X_b) \vee (X_u * X_c) \vee \Sigma X_u
\end{align*}
by the same argument as in the proof of Theorem \ref{splitting}. Therefore, we conclude that
\begin{align*}
&\Sigma \polyX{\lkK{v}} \vee \left(\polyX{\lkK{v}} * X_v \right) \vee \polyX{\dlK{v}} \\
\simeq & \Sigma((X_u * X_b) \vee (X_u * X_c) \vee \Sigma X_u) \\*
&\ \vee ((X_u * X_b) \vee (X_u * X_c) \vee \Sigma X_u) *X_v \\*
&\ \vee (X_u * X_a * X_b) \vee (X_u * X_c *X_d) \vee \Sigma X_u \\
\simeq & (X_u * X_a *X_b) \vee (X_u * X_b * X_v) \vee (X_u * X_v * X_c) \\*
&\  \vee (X_u * X_c * X_d) \vee \Sigma(X_u * X_b) \vee \Sigma(X_u * X_v) \vee \Sigma(X_u * X_c)  \\*
&\ \vee \Sigma X_u \vee \Sigma^2 X_u \\
\simeq &\mathcal{Z}^*_{K} (\underline{X}) \vee \Sigma X_u \vee \Sigma^2 X_u \\
\not\simeq &\mathcal{Z}^*_{K} (\underline{X}) .
\end{align*}
\end{remark}

Theorem \ref{s-vertex-decomposable splitting} and Theorem \ref{LnX homotopy} are obtained from Theorem \ref{splitting}. 
\begin{proof}[Proof of Theorem \ref{s-vertex-decomposable splitting}]
If $K$ is a simplex, then we have
\begin{align*}
\polyX{K} = X_1 * \cdots *X_m.
\end{align*}

We prove the theorem by induction on $m$. If $m=1$, then $K = \Delta^0$ and $\polyX{K} = X_1$.
Let $n$ be a positive integer. Assume that $\polyX{K'}$ can be decomposed as in Theorem \ref{s-vertex-decomposable splitting} for any s-vertex-decomposable simplicial complex $K'$ with $\card{V(K')} < n$.
Consider an s-vertex-decomposable simplicial complex $K$ with $\card{V(K)}=n$ which is not a simplex. 
Let $v, w$ be vertices of $K$ such that 
\begin{itemize}
\item $\dlK{v}$, $\lkK{v}$ are s-vertex-decomposable, and 
\item $\dlK{v} = \stK{w}$.
\end{itemize}

We have $\lkK{v} \neq \emptyset$ since $K$ is connected. Thus, by Theorem \ref{splitting}, we get 
\begin{align*}
\polyX{K} \simeq &\Sigma \polyX{\lkK{v}} \vee \left(\polyX{\lkK{v}} * X_v \right) \vee \polyX{\dlK{v}}  .
\end{align*}
Since $\dlK{v}=\stK{w}$ is connected and $\card{V(\dlK{v})} < \card{V(K)} =n$, we get from the assumption of the induction that 
\begin{align*}
\polyX{\dlK{v}} &\simeq \bigvee_{\beta \in B} \left( \Sigma^{s_\beta} X_1^{*l_{1, \beta}} * \cdots * X_m^{*l_{m,\beta}} \right)
\end{align*}
for a family $\{ (s_\beta, l_{1, \beta}, \ldots, l_{m, \beta}) \}_{\beta \in B}$.

If $\lkK{v}$ is connected, then there exists $\{ (r_\alpha, k_{1, \alpha}, \ldots, k_{m, \alpha}) \}_{\alpha \in A}$ such that
\begin{align*}
\polyX{\lkK{v}} &\simeq \bigvee_{\alpha \in A} \left( \Sigma^{r_\alpha} X_1^{*k_{1, \alpha}} * \cdots * X_m^{*k_{m,\alpha}} \right).
\end{align*}
Therefore, we obtain
\begin{align*}
\Sigma \polyX{\lkK{v}} \simeq &\bigvee_{\alpha \in A} \left( \Sigma^{r_\alpha +1} X_1^{*k_{1, \alpha}} * \cdots * X_m^{*k_{m,\alpha}} \right) ,\\
\polyX{\lkK{v}} * X_v \simeq &\bigvee_{\alpha \in A} \left( \Sigma^{r_\alpha} X_1^{*k_{1, \alpha}} * \cdots * X_v^{*(k_{v,\alpha} +1)} * \cdots * X_m^{*k_{m,\alpha}} \right) .
\end{align*}

If $\lkK{v}$ is not connected, then by Proposition \ref{non-connected}, $\lkK{v} = \stK{u} \sqcup \langle \{u'\} \rangle$ for some $u, u' \in V(K)$. Since $\stK{u}$ is connected and $\card{V(\stK{u})} < \card{V(\lkK{v})}$ $< n$, it follows from Proposition \ref{union decomposition} that
\begin{align*}
\polyX{\lkK{v}} &\simeq \left( \bigvee_{\alpha \in A} \left( \Sigma^{r_\alpha} X_1^{*k_{1, \alpha}} * \cdots * X_m^{*k_{m,\alpha}} \right) \right)\sqcup X_{u'}
\end{align*}
for a family $\{ (r_\alpha, k_{1, \alpha}, \ldots, k_{m, \alpha}) \}_{\alpha \in A}$. Then, by Proposition \ref{disjoint suspension}, we get
\begin{align*}
\Sigma \polyX{\lkK{v}} \simeq &\left( \bigvee_{\alpha \in A} \left( \Sigma^{r_\alpha +1} X_1^{*k_{1, \alpha}} * \cdots * X_m^{*k_{m,\alpha}} \right) \right) \vee \Sigma X_{u'} \vee S^1, \\
\polyX{\lkK{v}} * X_v \simeq &\left( \bigvee_{\alpha \in A} \left( \Sigma^{r_\alpha} X_1^{*k_{1, \alpha}} * \cdots * X_v^{*(k_{v,\alpha} +1)} * \cdots * X_m^{*k_{m,\alpha}} \right) \right)\\
&\ \vee (X_v * X_{u'}) \vee \Sigma X_v.
\end{align*}
Here, remark that a sphere $S^r$ can be written as 
\begin{align*}
S^r = \Sigma^{r-1} X_1^{*0} * \cdots * X_m^{*0}.
\end{align*}

In both cases, $\polyX{K}$ is denoted by
\begin{align*}
\polyX{K} \simeq \bigvee_{\gamma \in \Gamma} \left( \Sigma^{t_\gamma} X_1^{*j_{1, \gamma}} * \cdots * X_m^{*j_{m,\gamma}} \right).
\end{align*}
for some family $\{ (t_\gamma, j_{1, \gamma}, \ldots, j_{m, \gamma}) \}_{\gamma \in \Gamma}$ of tuples of non-negative integers, as desired.
\end{proof}

\begin{remark}
The above proof indicates that $r_{\alpha} + k_{1, \alpha} + \cdots + k_{m, \alpha} \geq 2$ for any $\alpha \in A$ if $m \geq 2$.
So, the wedge sum in Theorem \ref{s-vertex-decomposable splitting} is well-defined up to homotopy since each wedge summand is connected.
\end{remark}

\begin{proof}[Proof of Theorem \ref{LnX homotopy}]
We have
\begin{align*}
&I(L_1) = \langle \{1\} \rangle, & &I(L_2)=\langle \{1\}, \{2\} \rangle, & &I(L_3) = \langle \{1, 3\}, \{2\} \rangle.
\end{align*}
So, we get
\begin{align*}
&\polyX{I(L_1)} = X_1  , & &\polyX{I(L_2)} = X_1 \sqcup X_2, & &\polyX{I(L_3)} = (X_1*X_3) \sqcup X_2 .
\end{align*}
By setting $X_1=X_2=X_3=X$, we get the desired equality for $n=1,2,3$.

Let $n \geq 1$. 
The proof of Theorem \ref{forest s-v-d} indicates that $n+2$ is an s-shedding vertex of $I(L_{n+3})$. Furthermore, we have
\begin{align*}
\lk{I(L_{n+3})}{n+2} &= I(L_{n+3} \setminus \neib{L_{n+3}}{n+2} )= I(L_n) \neq \emptyset, \\
\dl{I(L_{n+3})}{n+2} &= I(L_{n+3} \setminus \{n+2\}) \\
&= I(L_{n+1} \sqcup \{n+3\}) = I(L_{n+1}) * \langle \{n+3\} \rangle.
\end{align*}
So, by Theorem \ref{splitting}, we obtain
\begin{align}
\label{recursive}
\polyIX{n+3} \simeq &\Sigma \polyIX{n} \vee (\polyIX{n} * X) \vee (\polyIX{n+1} * X) .
\end{align}

Define $Y_n$ for $n \geq 1$ by 
\begin{align*}
Y_n = \bigvee_{k,r \in \mathbb{N}_{\geq 0}} \left( {\bigvee}_{N_n(k,r)} \Sigma^r X^{*k} \right) ,
\end{align*}
where
\begin{align*}
N_n(k,r) &= \binom{k+1}{n-2k-3r+1} \binom{k+r}{r}.
\end{align*}
We note that $N_n (k,r) >0$ for non-negative integers $k, r$ if and only if $0 \leq n-2k-3r+1 \leq k+1$ and $0 \leq r \leq k+r$. This inequality implies that
\begin{align*}
\max \left\{\frac{n-3r}{3}, 0 \right\} \leq  &k \leq \frac{n-3r+1}{2}, \\
0 \leq &r \leq \frac{n+1}{3}.
\end{align*}So, it follows that
\begin{align*}
Y_n = \bigvee_{0 \leq r \leq \frac{n+1}{3}} \left( \bigvee_{\max \left\{ \frac{n-3r}{3}, 0 \right\} \leq k \leq \frac{n-3r+1}{2}} 
\left( {\bigvee}_{N_n(k,r)} \Sigma^r X^{*k} \right) \right) .
\end{align*}

In order to complete the proof, it is sufficient to show that $\polyIX{n} \simeq Y_n$ for any $n \geq 4$. First, the explicit descriptions of $Y_1$, $Y_2$, and $Y_3$ are obtained as follows.
\begin{align*}
Y_1 &= \bigvee_{0 \leq r \leq \frac{1+1}{3}} \left( \bigvee_{\max \left\{ \frac{1-3r}{3}, 0 \right\} \leq k \leq \frac{2-3r}{2}} 
\left( {\bigvee}_{N_1(k,r)} \Sigma^r X^{*k} \right) \right)  \\
 &=  \bigvee_{\frac{1}{3} \leq k \leq \frac{2}{2}} \left( {\bigvee}_{N_1(k,0)} X^{*k} \right)  \\
&={\bigvee}_{N_1(1,0)} X \\
&={\bigvee}_{\binom{2}{0} \binom{1}{0}} X \\
&=X,
\end{align*}
\begin{align*}
Y_2 &= \bigvee_{0 \leq r \leq \frac{2+1}{3}} \left( \bigvee_{\max \left\{ \frac{2-3r}{3}, 0 \right\} \leq k \leq \frac{3-3r}{2}} 
\left( {\bigvee}_{N_2(k,r)} \Sigma^r X^{*k} \right) \right)  \\
&= \left( \bigvee_{\frac{2}{3} \leq k \leq \frac{3}{2}} \left( {\bigvee}_{N_2(k,0)} X^{*k} \right) \right)
\vee \left( \bigvee_{0 \leq k \leq \frac{0}{2}} \left( {\bigvee}_{N_3(k,1)} \Sigma X^{*k} \right) \right)\\
&=\left( {\bigvee}_{N_2(1,0)} X  \right) \vee \left( {\bigvee}_{N_2(0,1)} S^0 \right)\\
&=\left( {\bigvee}_{\binom{2}{1} \binom{1}{0}} X  \right) \vee \left( {\bigvee}_{\binom{1}{0} \binom{1}{1}} S^0 \right)\\
&=X \vee X \vee S^0,
\end{align*}
\begin{align*}
Y_3 &= \bigvee_{0 \leq r \leq \frac{3+1}{3}} \left( \bigvee_{\max \left\{ \frac{3-3r}{3}, 0 \right\} \leq k \leq \frac{4-3r}{2}} 
\left( {\bigvee}_{N_3(k,r)} \Sigma^r X^{*k} \right) \right)  \\
&= \left( \bigvee_{\frac{3}{3} \leq k \leq \frac{4}{2}} \left( {\bigvee}_{N_3(k,0)} X^{*k} \right) \right)
\vee \left( \bigvee_{0 \leq k \leq \frac{1}{2}} \left( {\bigvee}_{N_3(k,1)} \Sigma X^{*k} \right) \right)\\
&=\left( {\bigvee}_{N_3(1,0)} X  \right) \vee \left( {\bigvee}_{N_3(2,0)} X^{*2}  \right) \vee \left( {\bigvee}_{N_3(0,1)} S^0 \right)\\
&=\left( {\bigvee}_{\binom{2}{2} \binom{1}{0}} X  \right) \vee \left( {\bigvee}_{\binom{3}{0} \binom{2}{0}} X^{*2}  \right) \vee \left( {\bigvee}_{\binom{1}{1} \binom{1}{1}} S^0 \right)\\
&=X \vee X^{*2} \vee S^0.
\end{align*}

Though $Y_2 \neq \polyIX{2}$ and $Y_3 \neq \polyIX{3}$, we obtain from Proposition \ref{disjoint suspension}, \ref{join wedge} that
\begin{align*}
\Sigma \polyIX{2} &= \Sigma (X \sqcup X) \simeq \Sigma X \vee \Sigma X \vee S^1 = \Sigma Y_2, \\
\polyIX{2} * X &= (X \sqcup X) *X \simeq (X *X) \vee (X *X) \vee \Sigma X = Y_2 *X, \\
\Sigma \polyIX{3} &= \Sigma (X^{*2} \sqcup X) \simeq  \Sigma X^{*2} \vee \Sigma X \vee S^1 = \Sigma Y_3, \\
\polyIX{3} * X &= (X^{*2} \sqcup X) *X \simeq (X^{*2} *X) \vee (X *X) \vee \Sigma X = Y_3 *X.
\end{align*}

Now assume that we have
\begin{align*}
\Sigma \polyIX{n} &\simeq \Sigma Y_n \simeq \bigvee_{k,r \in \mathbb{N}_{\geq 0}} \left( {\bigvee}_{\binom{k+1}{n-2k-3r+1} \binom{k+r}{r}} \Sigma^{r+1} X^{*k} \right), \\
\polyIX{n} *X &\simeq Y_n *X \simeq \bigvee_{k,r \in \mathbb{N}_{\geq 0}} \left( {\bigvee}_{\binom{k+1}{n-2k-3r+1} \binom{k+r}{r}} \Sigma^r X^{*(k+1)} \right), \\
\polyIX{n+1} *X &\simeq Y_{n+1} *X \simeq \bigvee_{k,r \in \mathbb{N}_{\geq 0}} \left( {\bigvee}_{\binom{k+1}{n-2k-3r+2} \binom{k+r}{r}} \Sigma^r X^{*(k+1)} \right)
\end{align*}
for some $n \geq 1$. Then, by (\ref{recursive}), we obtain
\begin{align*}
&\polyIX{n+3} \\
\simeq &\Sigma \polyIX{n} \vee (\polyIX{n} * X) \vee (\polyIX{n+1} * X)  \\
\simeq &\left(\bigvee_{k,r \in \mathbb{N}_{\geq 0}} \left( {\bigvee}_{\binom{k+1}{n-2k-3r+1} \binom{k+r}{r}} \Sigma^{r+1} X^{*k} \right) \right) \\*
&\ \vee \left(\bigvee_{k,r \in \mathbb{N}_{\geq 0}} \left( {\bigvee}_{\binom{k+1}{n-2k-3r+1} \binom{k+r}{r}} \Sigma^r X^{*(k+1)} \right) \right) \\*
&\ \vee \left(\bigvee_{k,r \in \mathbb{N}_{\geq 0}} \left( {\bigvee}_{\binom{k+1}{n-2k-3r+2} \binom{k+r}{r}} \Sigma^r X^{*(k+1)} \right) \right) \\
\simeq &\left(\bigvee_{k,r \in \mathbb{N}_{\geq 0}} \left( {\bigvee}_{\binom{k+1}{n-2k-3r+1} \binom{k+r}{r}} \Sigma^{r+1} X^{*k} \right) \right) \\*
&\ \vee \left(\bigvee_{k,r \in \mathbb{N}_{\geq 0}} \left( {\bigvee}_{\binom{k+1}{n-2k-3r+1} \binom{k+r}{r} + \binom{k+1}{n-2k-3r+2} \binom{k+r}{r}} \Sigma^r X^{*(k+1)} \right) \right) \\
= &\left(\bigvee_{k,r \in \mathbb{N}_{\geq 0}} \left( {\bigvee}_{\binom{k+1}{n-2k-3r+1} \binom{k+r}{r}} \Sigma^{r+1} X^{*k} \right) \right) \\*
&\ \vee \left(\bigvee_{k,r \in \mathbb{N}_{\geq 0}} \left( {\bigvee}_{\binom{k+2}{n-2k-3r+2} \binom{k+r}{r} } \Sigma^r X^{*(k+1)} \right) \right) \\
= &\left(\bigvee_{k \in \mathbb{N}_{\geq 0}, s=r+1 \in \mathbb{N} }\left( {\bigvee}_{\binom{k+1}{n-2k-3(s-1)+1} \binom{k+(s-1)}{s-1}} \Sigma^{s} X^{*k} \right) \right) \\*
&\ \vee \left(\bigvee_{l=k+1 \in \mathbb{N},r \in \mathbb{N}_{\geq 0}} \left( {\bigvee}_{\binom{(l-1)+2}{n-2(l-1)-3r+2} \binom{(l-1)+r}{r} } \Sigma^r X^{*l} \right) \right) \\
\simeq &\left(\bigvee_{k,r \in \mathbb{N}_{\geq 0}} \left( {\bigvee}_{\binom{k+1}{n-2k-3r+4} \binom{k+r-1}{r-1} + \binom{k+1}{n-2k-3r+4} \binom{k+r-1}{r} } \Sigma^{r} X^{*k} \right) \right) \\
= &\left(\bigvee_{k,r \in \mathbb{N}_{\geq 0}} \left( {\bigvee}_{\binom{k+1}{(n+3)-2k-3r+1} \binom{k+r}{r}} \Sigma^{r} X^{*k} \right) \right) \\
=& Y_{n+3} .
\end{align*}
Hence, by induction on $n$, we conclude that we have $\polyIX{n} \simeq Y_n$ for any $n \geq 4$, as desired.
\end{proof}

\section{The Independence Complex of the Lexicographic Product over a Forest}
The lexicographic product $\lex{G}{H}$ of $H$ over $G$, which is called {\it composition} by Harary \cite{Harary69}, is defined as follows.
\begin{definition}
Let $G, H$ be graphs. The {\it lexicographic product} $\lex{G}{H}$ of $H$ over $G$ is a graph defined by
\begin{align*}
&V(\lex{G}{H}) = V(G) \times V(H) ,\\
&E(\lex{G}{H}) = \left\{ (u_1, v_1)(u_2, v_2) \ \middle| \ 
\begin{aligned}
&u_1 u_2 \in E(G) \\
&\text{ or} \\
&u_1=u_2, v_1 v_2 \in E(H) 
\end{aligned}
\right\}.
\end{align*}
\end{definition}

The independence complex of $\lex{G}{H}$ and the polyhedral join are related in the following way.
\begin{proposition}
\label{lexicographic product polyhedral join}
Let $G, H$ be graphs. Then we have
\begin{align*}
I(\lex{G}{H}) = \mathcal{Z}^*_{I(G)} ( I(H), \emptyset ) .
\end{align*}
\end{proposition}
\begin{proof}
Follows from \cite[Example 25, Proposition 27]{Okura22}.
\end{proof}

\begin{proposition}
\label{simplicial topological relation}
Let $M$ be a simplicial complex on $[m]$ and $\pair = \{(K_i, L_i)\}_{i \in [m]}$ be a family of pairs of simplicial complexes and their subcomplexes.
For a family $\underline{X} = \{ X_v \}_{v \in \bigsqcup_{i \in [m]} V(K_i)}$ of finite CW complexes, we have
\begin{align*}
\polyX{\poly} \cong \mathcal{Z}^*_M( \polyX{\underline{K}}, \polyX{\underline{L}}) ,
\end{align*}
where $(\polyX{\underline{K}}, \polyX{\underline{L}} ) = \{(\polyX{K_i} , \polyX{L_i})\}_{i \in [m]}$.
\end{proposition}
\begin{proof}
Using Proposition \ref{join decomposition}, we obtain
\begin{align*}
&\polyX{\poly} \\ 
= &\bigcup_{\sigma \in \poly} (\underline{X}, \emptyset)^{* \sigma} \\
= &\bigcup_{S \in M} \left( \bigcup_{\sigma \in (\underline{K}, \underline{L})^{*S} } (\underline{X}, \emptyset)^{* \sigma} \right)\\
= &\bigcup_{S \in M} \polyX{\pair^{*S}} \\
= &\bigcup_{S \in M}  \polyX{J_1 * \cdots * J_m}  & &\left(J_i = \left\{
\begin{aligned}
& K_i & &(i \in S) \\
& L_i & &(i \notin S)
\end{aligned} \right. \right)\\
\cong &\bigcup_{S \in M} \left( \polyX{J_1} * \cdots * \polyX{J_m} \right) \\
=& \bigcup_{S \in M} (\polyX{\underline{K}}, \polyX{\underline{L}})^{* S} \\
=& \mathcal{Z}^*_M( \polyX{\underline{K}}, \polyX{\underline{L}}).
\end{align*}
\end{proof}

\begin{theorem}
\label{realization}
Let $M$ be a simplicial complex on $[m]$ and $\pair = \{(K_i, L_i)\}_{i \in [m]}$ be a family of pairs of simplicial complexes and their subcomplexes.
Then we have
\begin{align*}
\left| \poly \right| \cong \mathcal{Z}^*_M( |\underline{K}|, |\underline{L}|) ,
\end{align*}
where $| \underline{K}| = \{|K_i|\}_{i \in [m]}$, $|\underline{L}| = \{|L_i|\}_{i \in [m]}$.
\end{theorem}
\begin{proof}
Remark that for a simplicial complex $K$, we have
\begin{align*}
|K|=\mathcal{Z}^*_{K} ( \underline{\mathrm{pt}}) ,
\end{align*}
where $\underline{\mathrm{pt}}= \{X_v\}_{v \in V(K)}$, $X_v = \mathrm{pt}$ for any $v \in V(K)$.
So, this theorem follows from Proposition \ref{simplicial topological relation} by 
\begin{align*}
\left| \poly \right| = \mathcal{Z}^*_{\poly}(\underline{\mathrm{pt}}) \cong \mathcal{Z}^*_M \left( \mathcal{Z}^*_{\underline{K}} (\underline{\mathrm{pt}}), \mathcal{Z}^*_{\underline{L}}(\underline{\mathrm{pt}}) \right)= \mathcal{Z}^*_M( |\underline{K}|, |\underline{L}|).
\end{align*}
\end{proof}

The homeomorphism (\ref{lexico poly}) in Section \ref{introduction} is obtained from Proposition \ref{lexicographic product polyhedral join} and Theorem \ref{realization} by
\begin{align*}
|I(\lex{G}{H})| = |\mathcal{Z}^*_{I(G)} ( I(H), \emptyset ) | \cong \mathcal{Z}^*_{I(G)} ( |I(H)| ) .
\end{align*}
\noindent
Using this homeomorphism, we deduce Theorem \ref{forest wedge of spheres} and Theorem \ref{line theorem} from Theorem \ref{s-vertex-decomposable splitting} and Theorem \ref{LnX homotopy}, respectively.
\begin{proof}[Proof of Theorem \ref{forest wedge of spheres}]
Remark that $uv \in E(G)$ for any $v \in V(G) \setminus \{u\}$ means that $u$ is isolated in $I(G)$. Therefore, by Theorem \ref{forest s-v-d} and Proposition \ref{non-connected}, the assumption on $G$ implies that $I(G)$ is connected s-vertex-decomposable simplicial complex.
Thus, by Theorem \ref{s-vertex-decomposable splitting} and the homeomorphism (\ref{lexico poly}), we have
\begin{align*}
|I(\lex{G}{H})| &\cong \mathcal{Z}^*_{I(G)}( |I(H)|) \\
&\simeq \bigvee_{\alpha \in A} \left( \Sigma^{r_\alpha} |I(H)|^{*k_\alpha} \right)
\end{align*}
for some family $\{(r_{\alpha}, k_{\alpha})\}_{\alpha \in A}$ of pairs of non-negative integers.
If $|I(H)|$ is homotopy equivalent to a wedge sum of spheres, then so is $\Sigma^{r_\alpha} |I(H)|^{*k_\alpha}$. This is the desired conclusion.
\end{proof}

\begin{proof}[Proof of Theorem \ref{line theorem}]
By Theorem \ref{LnX homotopy} and the homeomorphism (\ref{lexico poly}), we obtain
\begin{align*}
|I(\lex{L_m}{H})| &\cong \mathcal{Z}^*_{I(L_m)} ( | I(H) | ) \\
&\simeq \left\{
\begin{aligned}
& |I(H)| & &(m=1), \\
&|I(H)| \sqcup |I(H)|& &(m=2) ,\\
&|I(H)|^{*2} \sqcup |I(H)| & &(m=3), \\
&\bigvee_{p, r \in \mathbb{N}_{\geq 0}} \left( {\bigvee}_{\binom{p+1}{m-2p-3r+1} \binom{p+r}{r}} \Sigma^r |I(H)|^{*p} \right) & &(m \geq 4).
\end{aligned} \right.
\end{align*}
Here, we have
\begin{align*}
|I(H)| &\simeq  {\bigvee}_{n} S^k ,\\
|I(H)| \sqcup |I(H)| & \simeq \left( {\bigvee}_n \sphere{k} \right) \sqcup \left( {\bigvee}_n \sphere{k} \right), \\
|I(H)|^{*2} \sqcup |I(H)| &\simeq \left( \left( {\bigvee}_n \sphere{k} \right) * \left( {\bigvee}_n \sphere{k} \right) \right) \sqcup \left( {\bigvee}_n \sphere{k} \right)  \\
&\simeq \left( {\bigvee}_{n^2} \sphere{2k+1} \right) \sqcup \left( {\bigvee}_n \sphere{k} \right) , \\
\Sigma^r |I(H)|^{*p} &\simeq \Sigma^r \left( {\bigvee}_{n} \sphere{k} \right)^{*p} \\
&\simeq  {\bigvee}_{n^p} S^{p(k+1)-1+r} .
\end{align*}
So, we get the desired formula for $m=1,2,3$. For $m \geq 4$, we get
\begin{align*}
&|I(\lex{L_m}{H})| \\
\simeq & \bigvee_{p,r \in \mathbb{N}_{\geq 0}} \left( {\bigvee}_{\binom{p+1}{m-2p-3r+1} \binom{p+r}{r}} \Sigma^r |I(H)|^{*p} \right) \\
\simeq & \bigvee_{p,r \in \mathbb{N}_{\geq 0}} \left( {\bigvee}_{\binom{p+1}{m-2p-3r+1} \binom{p+r}{r}} \left( {\bigvee}_{n^p} S^{p(k+1)-1+r} \right) \right) \\
\simeq & \bigvee_{p,r \in \mathbb{N}_{\geq 0}} \left( {\bigvee}_{n^p \binom{p+1}{m-2p-3r+1} \binom{p+r}{r}} S^{p(k+1)-1+r} \right) \\
\simeq & \bigvee_{p \in \mathbb{N}_{\geq 0}} \left(\bigvee_{d=p(k+1)-1+r \geq 0} \left( {\bigvee}_{n^p \binom{p+1}{m-2p-3(d-p(k+1)+1)+1} \binom{p+(d-p(k+1)+1)}{p}} S^d \right) \right) \\
\simeq &\bigvee_{d \geq 0} \left( {\bigvee}_{\sum_{p \geq 0} n^p \binom{p+1}{3(d-pk+1)-m} \binom{d-pk+1}{p} } \sphere{d} \right) 
\end{align*}
as desired.
\end{proof}

\begin{example}
Kozlov \cite[Proposition 5.2]{Kozlov99} proved that
\begin{align*}
|I(C_n)| &\simeq \left\{
\begin{aligned}
&\sphere{k - 1} \vee \sphere{k - 1}  & &(n =3k), \\
&\sphere{k-1} & &(n =3k+1), \\
&\sphere{k} & &(n =3k+2)  .
\end{aligned} \right. 
\end{align*}
So, we can determine the homotopy types of $|I(\lex{L_m}{C_n})|$ for any $m \geq 1$ and $n \geq 3$ by Theorem \ref{line theorem}.
\end{example}



\begin{thebibliography}{10}

\bibitem{Adamaszek12}
Micha{\l} Adamaszek,
\newblock Splittings of independence complexes and the powers of cycles,
\newblock {\em Journal of Combinatorial Theory, Series A}, 119:1031--1047,
  2012.

\bibitem{Ayzenberg13}
A.~A. Ayzenberg,
\newblock Substitutions of polytopes and of simplicial complexes, and
  multigraded {B}etti numbers,
\newblock {\em Transactions of the Moscow Mathematical Society}, 74:175--202,
  2013.

\bibitem{BBCG10}
A.~Bahri, M.~Bendersky, F.R. Cohen, and S.~Gitler,
\newblock The polyhedral product functor: A method of decomposition for
  moment-angle complexes, arrangements and related spaces,
\newblock {\em Advances in Mathematics}, 225:1634--1668, 2010.

\bibitem{Barmak13}
Jonathan~Ariel Barmak,
\newblock Star clusters in independence complexes of graphs,
\newblock {\em Advances in Mathematics}, 241:33--57, 2013.

\bibitem{BjornerWachs97}
Anders Bj{\"{o}}rner and Michelle~L. Wachs,
\newblock Shellable nonpure complexes and posets. {II},
\newblock {\em Transactions of the American Mathematical Society},
  349(10):3945--3975, 1997.

\bibitem{BousquetmelouLinussonNevo08}
Mireille Bousquet-M{\'{e}}lou, Svante Linusson, and Eran Nevo,
\newblock On the independence complex of square grids,
\newblock {\em Journal of Algebraic Combinatorics}, 27:423--450, 2008.

\bibitem{Chari00}
Manoj~K. Chari,
\newblock On discrete {M}orse functions and combinatorial decompositions,
\newblock {\em Discrete Mathematics}, 217:101--113, 2000.

\bibitem{Engstrom09}
Alexander Engstr{\"{o}}m,
\newblock Complexes of directed trees and independence complexes,
\newblock {\em Discrete Mathematics}, 309:3299--3309, 2009.

\bibitem{Forman98}
Robin Forman,
\newblock {M}orse theory for cell complexes,
\newblock {\em Advances in Mathematics}, 134:90--145, 1998.

\bibitem{Harary69}
Frank Harary,
\newblock {\em Graph Theory},
\newblock Addison-Wesley Publishing Company, 1969.

\bibitem{Hatcher01}
Allen Hatcher,
\newblock {\em Algebraic Topology},
\newblock Cambridge University Press, 2001.

\bibitem{Iriye12}
Kouyemon Iriye,
\newblock On the homotopy types of the independence complexes of grid graphs
  with cylindrical identification,
\newblock {\em Kyoto Journal of Mathematics}, 52(3):479--501, 2012.

\bibitem{Kawamura10}
Kazuhiro Kawamura,
\newblock Independence complexes of chordal graphs,
\newblock {\em Discrete Mathematics}, 310:2204--2211, 2010.

\bibitem{Kozlov99}
Dmitry~N. Kozlov,
\newblock Complexes of directed trees,
\newblock {\em Journal of Combinatorial Theory, Series A}, 88:112--122, 1999.

\bibitem{Okura22}
Kengo Okura,
\newblock Shellability of polyhedral joins of simplicial complexes and its
  application to graph theory,
\newblock {\em The Electronic Journal of Combinatorics}, 29(3):\#P3.53, 2022.

\bibitem{Thapper08}
Johan Thapper,
\newblock Independence complexes of cylinders constructed from square and
  hexagonal grid graphs,
\newblock {\em arXiv e-prints}, arXiv:0812.1165, 2008.

\bibitem{VantuylVillarreal08}
Adam {Van Tuyl} and Rafael~H. Villarreal,
\newblock Shellable graphs and sequentially {C}ohen-{M}acaulay bipartite
  graphs,
\newblock {\em Journal of Combinatorial Theory, Series A}, 115:799--814, 2008.

\bibitem{VandermeulenVantuyl17}
Kevin~N. {Vander Meulen} and Adam {Van Tuyl},
\newblock Shellability, vertex decomposability, and lexicographical products of
  graphs,
\newblock {\em Contributions to Discrete Mathematics}, 12(2):63--68, 2017.

\bibitem{Woodroofe09}
Russ Woodroofe,
\newblock Vertex decomposable graphs and obstructions to shellability,
\newblock {\em Proceedings of the American Mathematical Society},
  137(10):3235--3246, 2009.

\end{thebibliography}
\end{document}